\documentclass[final,leqno]{siamltex}
\usepackage{epsfig}
\usepackage{amsmath}
\usepackage{amssymb}
\usepackage{tikz}
\usepackage{graphicx}
\usepackage[notcite,notref]{showkeys}

\newcommand{\bq}{{\bf q}}
\newcommand{\bn}{{\bf n}}

\def\T{{\mathcal T}}

\def\bn{{\bf n}}
\def\bq{{\bf q}}

  \def\b#1{\mathbf{#1}} 
\def\a#1{\begin{align*}#1\end{align*}} \def\an#1{\begin{align}#1\end{align}} \def\t#1{\hbox{#1}}

\newcommand{\pT}{{\partial T}}

\def\3bar{{|\hspace{-.02in}|\hspace{-.02in}|}}

\title{A $P_{k+2}$ polynomial lifting operator on polygons and polyhedrons}

\author{Xiu Ye\thanks{Department of
Mathematics, University of Arkansas at Little Rock, Little Rock, AR
72204 (xxye@ualr.edu). This research was supported in part by
National Science Foundation Grant DMS-1620016.}
\and
Shangyou Zhang\thanks{Department of
Mathematical Sciences, University of Delaware, Newark, DE 19716 (szhang@udel.edu).}
}

\begin{document}

\maketitle

\begin{abstract}
A $P_{k+2}$ polynomial lifting operator is defined on polygons and polyhedrons.
It lifts
   discontinuous polynomials inside the polygon/polyhedron and on the faces to a 
  one-piece $P_{k+2}$ polynomial.
With this lifting operator, we prove that the weak Galerkin finite element solution,
  after this lifting,  
   converges at two orders higher than the optimal order, in both $L^2$ and $H^1$ norms.
The theory is confirmed by numerical solutions of 2D and 3D Poisson equations. 
\end{abstract}



\section{Introduction} 

In weak Galerkin finite element methods \cite{wy,yz-sf-wg},  
    discontinuous polynomials, $u_0$ defined inside 
  each element and $u_b$ defined on each face of element,  are employed to form 
  an approximation space.
In particular, on triangular/tetrahedral grids, the $P_k$-$P_{k+1}$  ($P_k$ inside a triangle,
  $P_{k+1}$ on an edge) weak Galerkin finite element solution is two-order superconvergent 
 in both $L^2$ and $H^1$-like norms \cite{Al-Taweel}.
Further,  with a careful construction of weak gradient, such $P_k$-$P_{k+1}$ 
   weak Galerkin finite element is also two-order superconvergent on general 
   polygonal and polyhedral meshes \cite{ye-3}.
Here the super-convergence is defined for the difference between finite element solution 
  $u_0$ and the local $L^2$ projection $Q_hu$ of the exact solution.

In this paper,  we construct a $P_{k+2}$ polynomial lifting operator.
It lifts an $(n+1)$-piece polynomial, $\{u_0, u_b\}$, on a $n$-polygon/polyhedron $T$ to
    a one-piece $P_{k+2}$ polynomial on $T$.
After such a lifting/post-processing,  the
  weak Galerkin finite element solution is two-order super-convergent to
  the exact solution,  i.e.,
\a{  \|u - u_h\|_0 + h|u-u_h|_{1,h} & \le C h^{k+1} |u|_{k+1}, \\
     \|u - L_h u_h\|_0 + h|u-L_hu_h|_{1,h} & \le C h^{k+3} |u|_{k+3},  }
where $u_h$ and $u$ are the finite element solution and the exact solution, respectively,
  and $h$ is the mesh size.

This polynomial lifting operator is different from traditional polynomial lifting operators
  \cite{Ainsworth,Bernardi,Bernardi-Maday,Guo,Munoz-Sola}.
These operators
   only lift a polynomial trace on the boundary of an element to a polynomial inside the element,
  stably, i.e., subject to the minimum or a small energy.
But here we lift both trace data and interior data to a polynomial, subject to the
     $P_{k+2}$ accuracy.
Additionally, even the trace (of boundary polynomials) is discontinuous here.
Well, such a discontinuous-trace polynomial lifting is studied in
  \cite{Demkowicz1,Demkowicz2,Demkowicz3}, but for $H(\t{curl})$ and $H(\t{div})$
   polynomial lifting.

\section{Weak Galerkin finite element}

For solving a model Poisson equation, 
\begin{eqnarray}
-\Delta u&=&f\quad \mbox{in}\;\Omega,\label{pde}\\
u&=&0\quad\mbox{on}\;\partial\Omega,\label{bc}
\end{eqnarray}
where $\Omega$ is a polytopal domain in $\mathbb{R}^2$ or $\mathbb{R}^3$,
we subdivide the domain into shape-regular polygons/polyhedrons of size $h$,   ${\cal T}_h$.
For polynomial degree $k\ge 1$, we define the weak Galerkin finite
element spaces by
\begin{equation}\label{vhspace}
V_h=\{v_h=\{v_0,v_b\}:\; v_0|_T\in P_k(T),\ v_b|_e\in P_{k+1}(e),\ e\subset\pT,  T\in \T_h\}
\end{equation}
and  
\begin{equation}\label{vh0space}
V^0_h=\{v_h : \ v_h\in V_h,\  v_b=0 \mbox{ on } e\subset \partial\Omega\}.
\end{equation}
The weak Galerkin finite element function assumes one $d$-dimensional
    $P_k$ polynomial inside each element $T$,
 and one $(d-1)$-dimensional $P_{k+1}$ polynomial on each face edge/polygon $e$.

On an element $T\in\T_h$, we define the weak gradient $\nabla_w v_h$ of a weak function 
  $v_h=\{v_0,v_b\}\in V_h$ by the solution of polynomial equation on $T$:
\begin{equation}\label{d-d}
  \int_T \nabla_w v_h \bq d\b x=\int_{\partial T} v_b \bq\cdot\bn dS -\int_T v_0
           \nabla\cdot \bq d\b x \quad
   \forall \bq\in \Lambda_k(T),
\end{equation} where $\Lambda_k(T)$ is a piece-wise polynomial space, but with one piece
   polynomial divergence and one piece polynomial trace on each face, on 
a sub-triangular/tetrahedral subdivision of $T=\{ T_i, i=1,...,n\}$,  
\a{ \Lambda_k(T) =\{ \b q\in H(\t{div}, T) :   \  
   & \b q|_{T_i} \in P_{k+1}^d(T_i), T_i\subset T, \\  \  
             & \nabla\cdot \b q\in P_k(T), \ \b q\cdot \b n|_e \in P_{k+1}(e) \}.
        } Here $\b n$ is a fixed normal vector on edge/polygon $e$.
To get a simplicial subdivision on $T$,  some face edges/polygons have to be 
   subdivided.  That is, in addition to $T=\cup_{i} \overline{T}_i$,
   $e=\cup_{j} \overline{e}_j$, where $\{e_j\}$ is the set of face edges/triangles of
   $\{T_i\}$.

A weak Galerkin finite element approximation for (\ref{pde})-(\ref{bc}) is
  defined by the unique solution $u_h=\{u_0,u_b\}\in V_h^0$ satisfying
\begin{equation}\label{wg}
(\nabla_w u_h,\nabla_wv_h)=(f,\; v_0) \quad\forall v_h=\{v_0,v_b\}\in V_h^0.
\end{equation} 
In \cite{ye-3}, both \eqref{d-d} and \eqref{wg} are proved to have a unique solution.

\begin{theorem}(\cite{ye-3}) Let $u$ and $u_h$ be the solutions of \eqref{pde} and \eqref{wg},
   respectively.  The following two-order superconvergence holds
\an{ \label{o-2} \| Q_h u - u_h\|_0 + h \3bar Q_h u - u_h \3bar & \le
     C h^{k+3} |u|_{k+3},
}  where $Q_h u=\{ Q_0 u, Q_b u \}\in V_h^0$  ($Q_0$ and $Q_b$ are local
   $L^2$-projection on $T$ and $e$ respectively), and 
   $\3bar v_h \3bar=(\nabla_w v_h,\nabla_wv_h)^{1/2}$.
\end{theorem}

\section{A $P_{k+2}$ polynomial lifting operator}
On an $m$-face polygon/polyhedron $T$  we have $(m+1)$ pieces of polynomials from
  a weak Galerkin finite element function.
We need to lift these polynomials to a one-piece $P_{k+2}$ polynomial, preserving 
  $P_{k+2}$ polynomials in the sense that $L_hQ_h u=u$ if $u$ is a $P_{k+2}$ polynomial.

\begin{theorem} The local $L^2$ projection $Q_h: u\in P_{k+2}(T) \to u_h=\{Q_0 u, Q_b u\}
  \in V_h $ is an injection, i.e.,
\a{   Q_h u = 0 \t{ if and only if } u=0.  }  
\end{theorem}

\begin{proof}  
Let $u\in P_{k+2}(T)$ and $Q_h u=0$.
For any vector polynomial $\b q_{k+1} \in [ P_{k+1}(T]^2$,
  we have
\a{ \int_T \nabla u \cdot \b q_{k+1} d\b x &=
      \sum_{e\subset \partial T} \int_{e} u \b q_{k+1}\cdot \b n dS
          - \int_T   u \nabla \cdot \b q_{k+1} d\b x \\
   &= \sum_{e\subset \partial T} \int_{e} Q_b u \b q_{k+1}\cdot \b n dS
          - \int_T   Q_0 u \nabla \cdot \b q_{k+1} d\b x \\
     &= 0. }
Thus $\nabla u=\b 0$ everywhere and $u=C$.
Since $Q_0 u=0$, $C=0$ and $u=0$.
\end{proof}
 
\begin{theorem} The $P_{k+2}$ polynomial lifting operator $L_h$, defined in
  \eqref{lift} below, is $P_{k+2}$ polynomial preserving in the sense that
\an{\label{p-k-2}  L_h Q_h u = u, \quad \t{ if }  u \in P_{k+2}(T).  } 
Consequently we have
\an{ \label{optimal} \|u-L_h Q_h u \|_0 + h|u-L_h Q_h u  |_{1,h} \le C h^{k+3} | u|_{k+3}. }
\end{theorem}

\begin{proof} Let $P_h: u_h=\{u_0, u_b\} \in V_h \to  \{ (P_hu_h)_0, (P_hu_h)_b\} \in 
     V_h$ be the local, discrete $L^2(T)$ projection on to the
  image space $Q_h P_{k+2}(T)$, i.e.,  
\a{  &\quad \ \int_T (P_hu_h)_0 Q_0 p_{k+2} d\b x 
        + \sum_{e\subset \partial T} \int_{e} (P_hu_h)_b  Q_b p_{k+2} dS \\
     &=
      \int_T  u_0 Q_0 p_{k+2} d\b x 
        + \sum_{e\subset \partial T} \int_{e} u_b  Q_b p_{k+2} dS \quad \forall p_{k+2}(T).  }
The above equation has a unique solution as the left hand side bilinear form is coercive.
By last theorem,  $Q_h$ is one-to-one from $P_{k+2}(T)$ on to the image space $P_hV_h$.
Its inverse defines an unique lifting operator:
\an{ \label{lift}  L_h u_h = Q_h^{-1}(P_h u_h) \in \prod_{T\in \T_h} P_{k+2}(T).   }
By definition, \eqref{p-k-2} holds.  Further, because $L_h Q_h $ is a stable,
  local preserving $P_{k+2}$ polynomial operator,  by \cite{Scott-Zhang},
  it is an optimal-order interpolation operator and \eqref{optimal} holds.
\end{proof}

\begin{theorem} Let $u$ and $u_h$ be the solutions of \eqref{pde} and \eqref{wg},
   respectively.  Then
\a{ \|  u - L_h u_h\|_0 + h | u - L_h u_h |_{1,h} & \le
     C h^{k+3} |u|_{k+3},
}  where  $| u |_{1,h}^2=\sum_{T\in \T_h}(\nabla u,\nabla u)$.
\end{theorem}

\begin{proof} Noting the weak gradient of $u_h-P_h u_h$ is a piece-wise higher order,
   $\3bar\cdot\3bar$-orthogonal  polynomial
  over the polynomial $\nabla L_h u_h$,  we have 
\a{ |L_h u_h  |_{1,h}^2= \3bar P_h u_h \3bar^2=\3bar  u_h \3bar^2-\3bar (I-P_h) u_h \3bar^2 \le
             \3bar u_h \3bar^2.  } 
By the triangle inequality, \eqref{optimal} and \eqref{o-2}, 
\a{  | u - L_h u_h |_{1,h} &\le | u - L_h Q_h u |_{1,h}+ | L_h (Q_h u - u_h) |_{1,h}\\ 
     &\le C h^{k+3} | u |_{k+3}+  \3bar Q_h u - u_h  \3bar\\
     &\le C h^{k+2} | u |_{k+3}. }
 By the finite dimensional norm equivalence with scaling, 
   the trace inequality and the definition of weak gradient,
  we have
\a{   \|  L_h   u_h \|_0^2  & \le C  \sum_{T\in \T_h}
            \Big(  \| P_h u_0 \|_T^2 
           + 2 h \| P_h( u_0 -  u_b) \|_{\partial T}^2  
           \Big)  \\ &  \le  C h \| L_h   u_h \|_0 \3bar (I-P_h) u_h \3bar. } 
By the triangle inequality, \eqref{optimal} and \eqref{o-2},  we get 
\a{  \|  u - L_h u_h\|_0 & \le  \|  u - L_h Q_h u \|_0+
              \|  L_h (Q_h u -  u_h) \|_0  \\
              &\le C h^{k+3} | u |_{k+3}. }  
\end{proof}

\section{Numerical Experiments}\label{Section:numerical-experiments}

\begin{figure}[htb]\begin{center}
\includegraphics[width=1.2in]{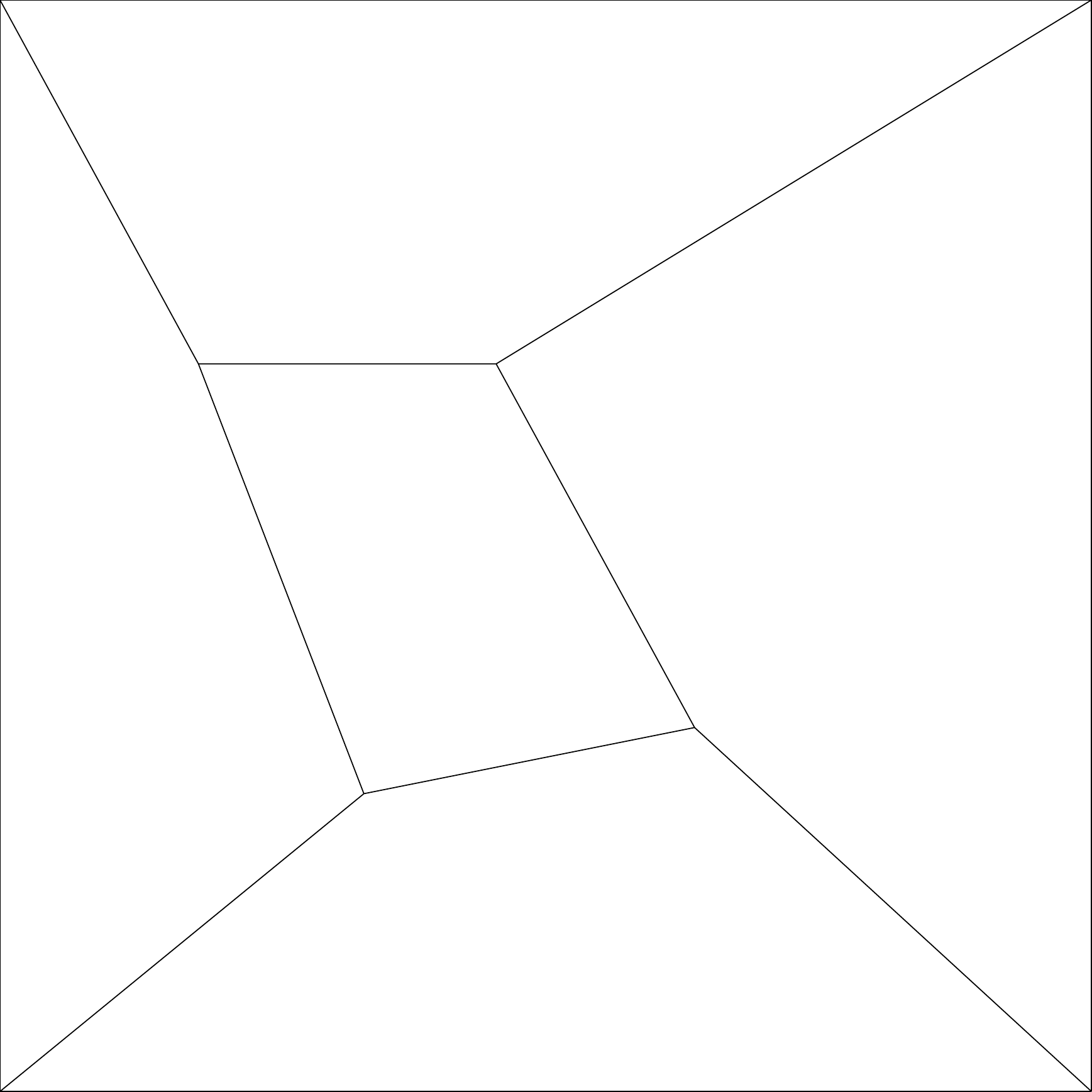} \
\includegraphics[width=1.2in]{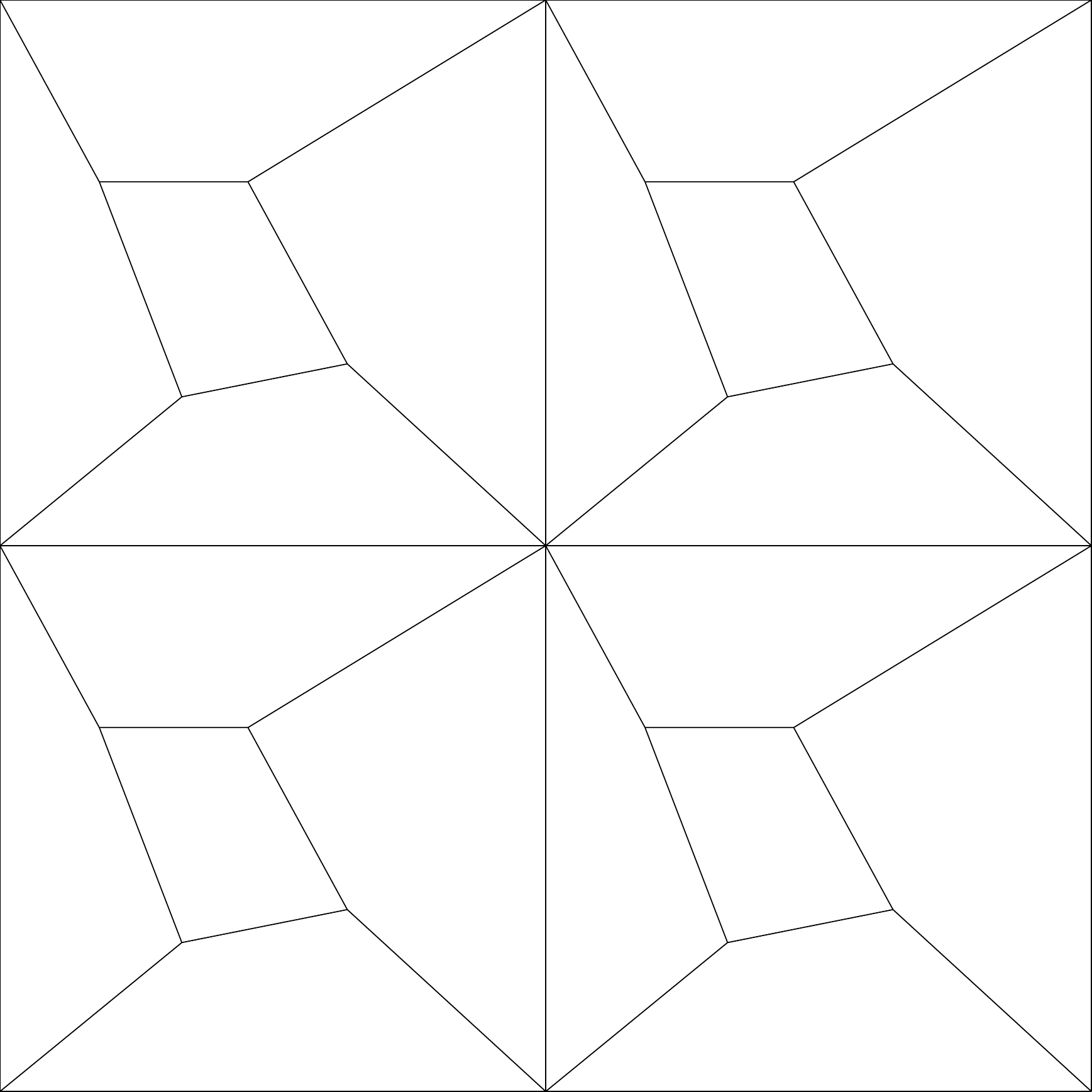} \
\includegraphics[width=1.2in]{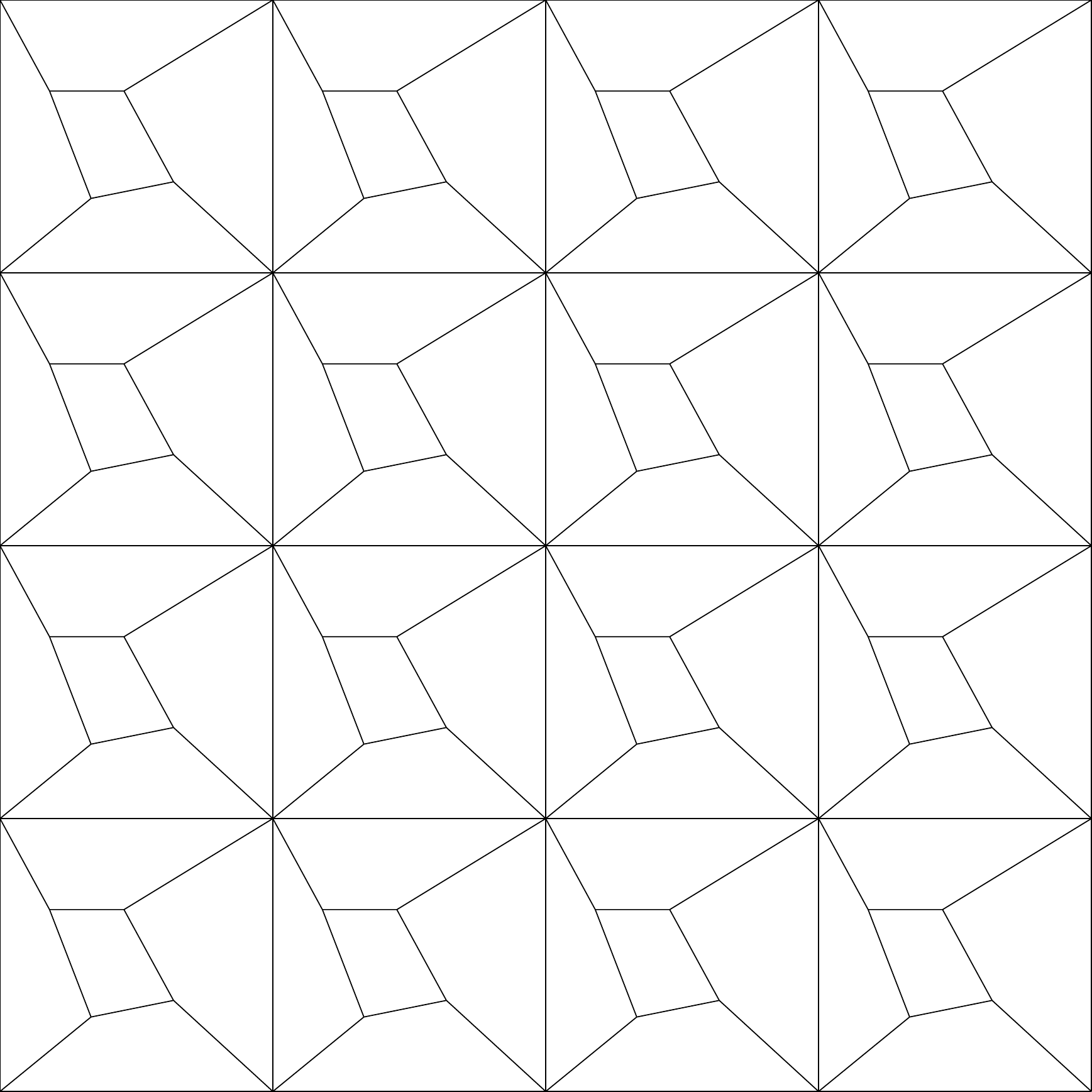}

\caption{The first three levels of quadrilateral grids, for Table \ref{t1}.  }
\label{g-4}
\end{center}
\end{figure}

\begin{table}[ht]
  \centering   \renewcommand{\arraystretch}{1.05}
  \caption{ Errors and orders of convergence by the $P_1$-$P_2$ WG finite element on quadrilateral
      grids shown in Figure \ref{g-4} for \eqref{s-1}. }
\label{t1}
\begin{tabular}{c|cc|cc|cc}
\hline
level     & $\| u-  u_h \|_0 $  &rate &  $\| Q_h u- u_h \|_0 $ &rate  &
        $\| u- L_h u_h \|_0 $ &rate   \\ \hline
 5&   0.7356E-03 &  2.00&   0.9360E-06 &  4.00&   0.1308E-05 &  4.00 \\
 6&   0.1838E-03 &  2.00&   0.5851E-07 &  4.00&   0.8178E-07 &  4.00 \\
 7&   0.4595E-04 &  2.00&   0.3663E-08 &  4.00&   0.5116E-08 &  4.00 \\
\hline
& $ | u-  u_h  |_{1,h} $  &rate &  $\3bar Q_h u- u_h \3bar $ &rate &
        $| u- L_h u_h |_{1,h} $ &rate   \\
\hline
 5&   0.5049E-01 &  1.00&   0.2156E-03 &  3.00&   0.2101E-03 &  3.00 \\
 6&   0.2524E-01 &  1.00&   0.2696E-04 &  3.00&   0.2627E-04 &  3.00 \\
 7&   0.1262E-01 &  1.00&   0.3371E-05 &  3.00&   0.3284E-05 &  3.00 \\
\hline  
\end{tabular}%
\end{table}%

\begin{table}[ht]
  \centering   \renewcommand{\arraystretch}{1.05}
  \caption{ Errors and orders of convergence by the $P_2$-$P_3$ WG finite element on quadrilateral
      grids shown in Figure \ref{g-4} for \eqref{s-1}. }
\label{t12}
\begin{tabular}{c|cc|cc|cc}
\hline
level     & $\| u-  u_h \|_0 $  &rate &  $\| Q_h u- u_h \|_0 $ &rate  &
        $\| u- L_h u_h \|_0 $ &rate   \\ \hline
 4&   0.2229E-03 &  3.00&   0.7659E-06 &  4.98&   0.8555E-06 &  4.98 \\
 5&   0.2787E-04 &  3.00&   0.2404E-07 &  4.99&   0.2682E-07 &  5.00 \\
 6&   0.3484E-05 &  3.00&   0.7521E-09 &  5.00&   0.8390E-09 &  5.00 \\
\hline
& $ | u-  u_h  |_{1,h} $  &rate &  $\3bar Q_h u- u_h \3bar $ &rate &
        $| u- L_h u_h |_{1,h} $ &rate   \\
\hline
 4&   0.1293E-01 &  2.00&   0.1487E-03 &  3.99&   0.9441E-04 &  3.99 \\
 5&   0.3233E-02 &  2.00&   0.9307E-05 &  4.00&   0.5911E-05 &  4.00 \\
 6&   0.8084E-03 &  2.00&   0.5819E-06 &  4.00&   0.3696E-06 &  4.00 \\
\hline  
\end{tabular}%
\end{table}%

We solve the 2D Poisson equation \eqref{pde} on the unit square domain. The exact
  solution is chosen as
\an{ \label{s-1} u=\sin(\pi x)\sin(\pi y).
  }
We compute the solution \eqref{s-1} on a perturbed quadrilateral grids, shown
   in Figure \ref{g-4}. 
We have two orders of superconvergence in $L^2$-norm and in $H^1$-like norm,
  shown in Tables \ref{t1}-\ref{t12}. 
In particular,  the error after lifting is two orders higher than that of the original error.

Next we solve again the 2D Poisson equation \eqref{pde} on the unit square domain with
 exact solution \eqref{s-1}.
We use  quadrilateral-pentagon-hexagon hybrid grids, shown
   in Figure \ref{g-6}. 
Again the error after lifting is two orders higher, shown in Table \ref{t2}.

\begin{figure}[htb]\begin{center}
\includegraphics[width=1.2in]{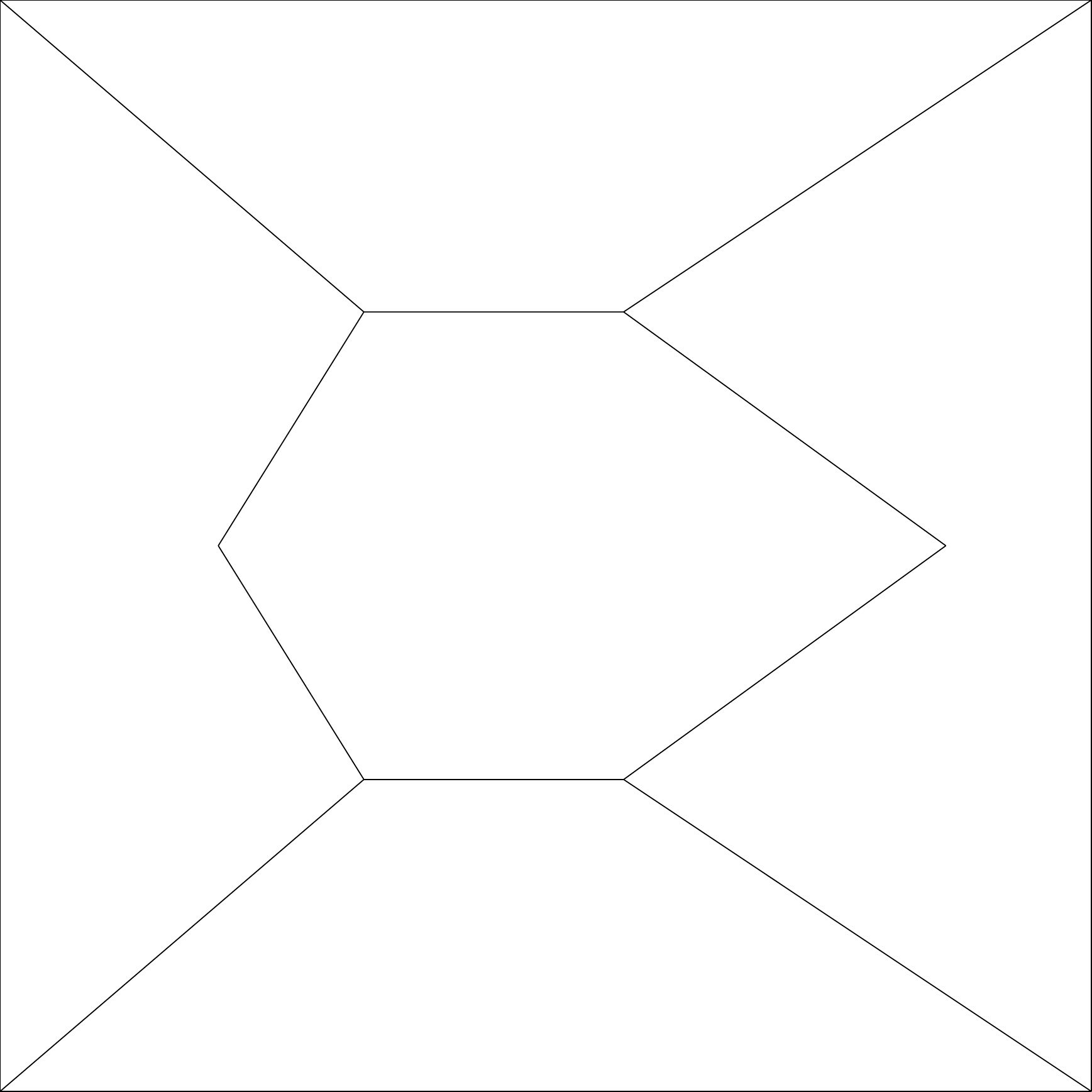} \
\includegraphics[width=1.2in]{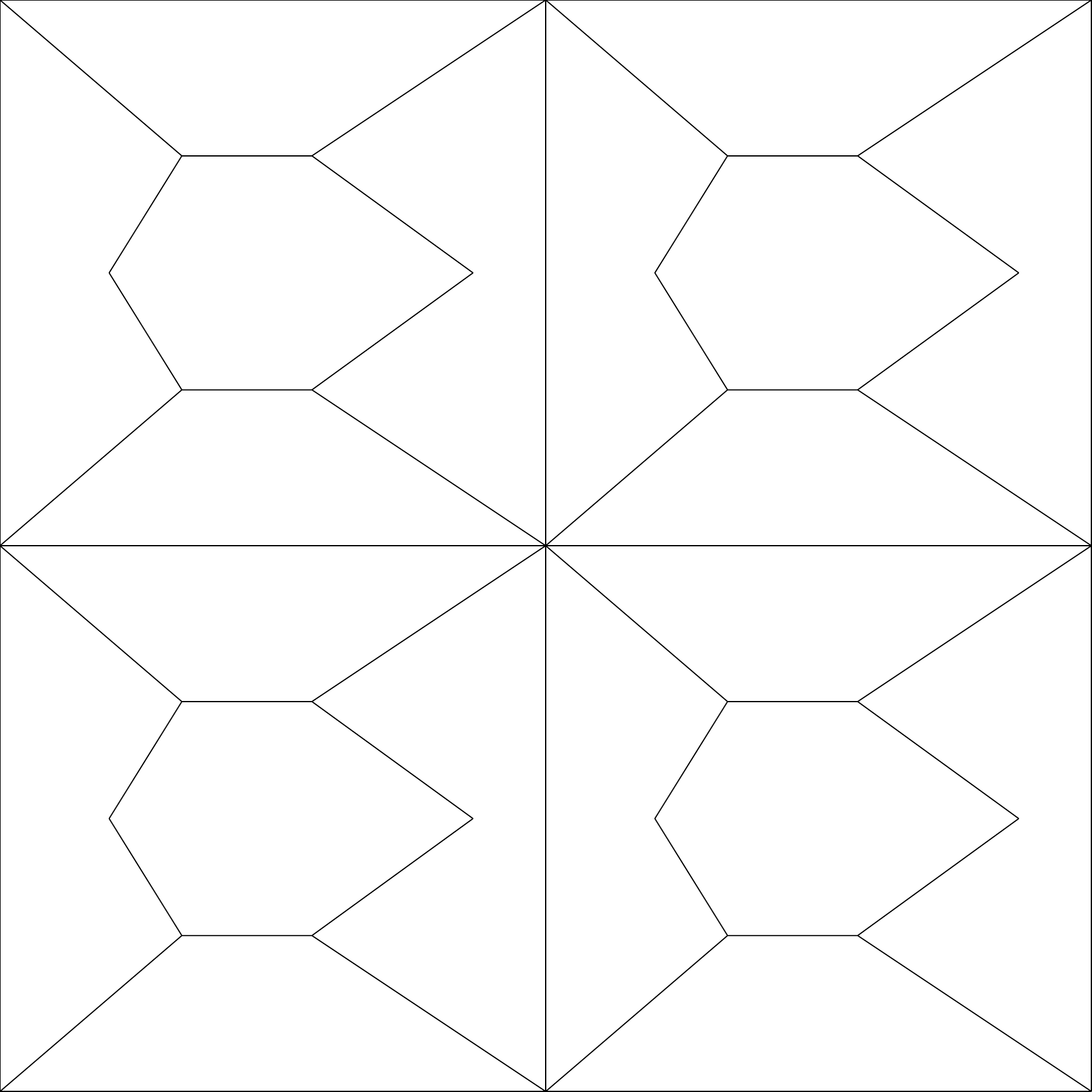} \
\includegraphics[width=1.2in]{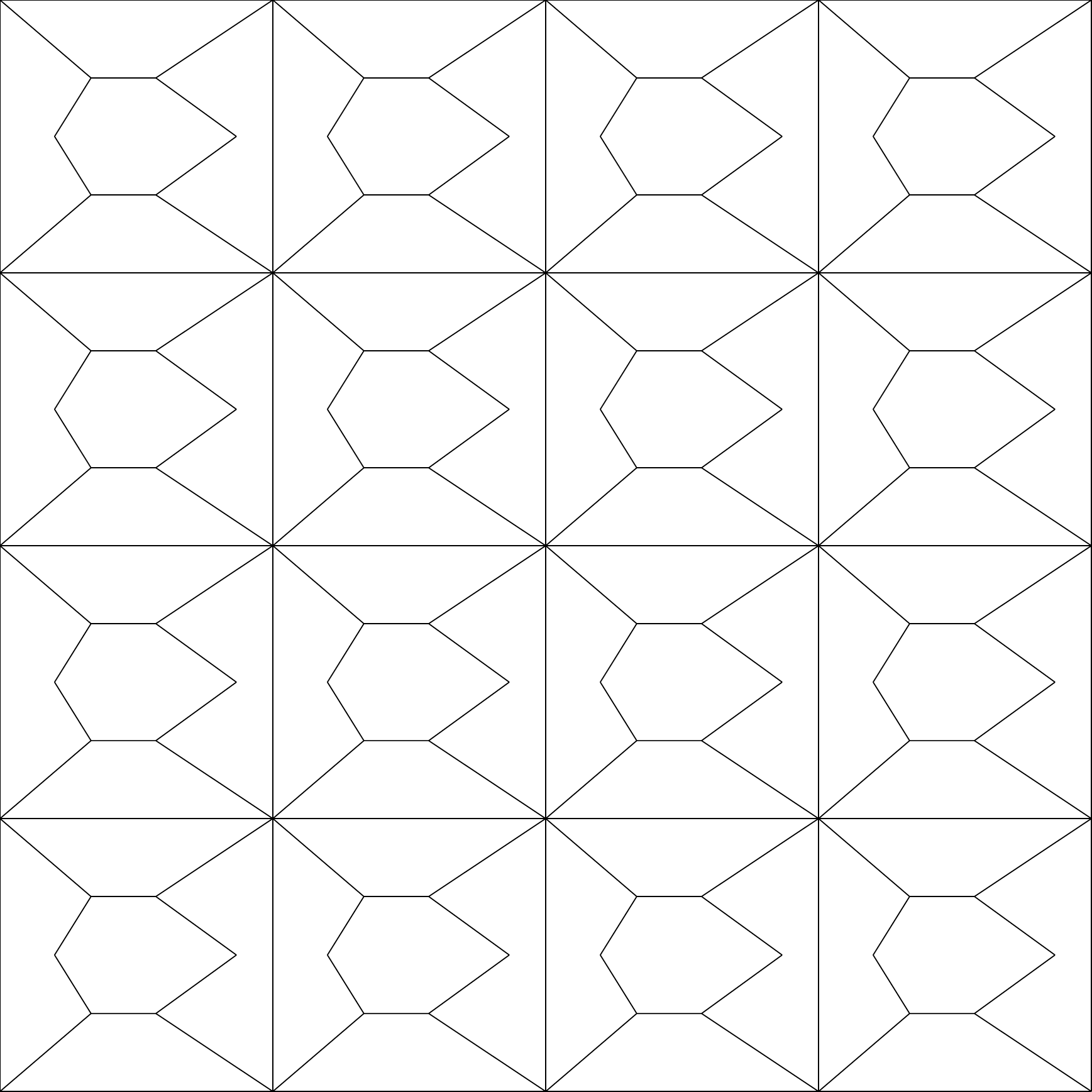}

\caption{The first three levels of mixed-polygon grids, for Tables \ref{t2}.  }
\label{g-6}
\end{center}
\end{figure}

\begin{table}[ht]
  \centering   \renewcommand{\arraystretch}{1.05}
  \caption{ Errors and orders of convergence, by the $P_1$-$P_2$ WG finite element on mixed-polygon grids shown in Figure \ref{g-6} for \eqref{s-1}. }
\label{t2}
\begin{tabular}{c|cc|cc|cc}
\hline
level     & $\| u-  u_h \|_0 $  &rate &  $\| Q_h u- u_h \|_0 $ &rate  &
        $\| u- L_h u_h \|_0 $ &rate   \\ \hline
 5&   0.8444E-03 &  2.00&   0.1504E-05 &  4.00&   0.1973E-05 &  4.00 \\
 6&   0.2110E-03 &  2.00&   0.9406E-07 &  4.00&   0.1234E-06 &  4.00 \\
 7&   0.5273E-04 &  2.00&   0.5875E-08 &  4.00&   0.7707E-08 &  4.00 \\
\hline
& $ | u-  u_h  |_{1,h} $  &rate &  $\3bar Q_h u- u_h \3bar $ &rate &
        $| u- L_h u_h |_{1,h} $ &rate   \\
\hline
 5&   0.5891E-01 &  1.00&   0.4798E-03 &  3.00&   0.3272E-03 &  3.00 \\
 6&   0.2945E-01 &  1.00&   0.6001E-04 &  3.00&   0.4090E-04 &  3.00 \\
 7&   0.1472E-01 &  1.00&   0.7502E-05 &  3.00&   0.5113E-05 &  3.00 \\
\hline  
\end{tabular}%
\end{table}%

Finally we solve the 3D Poisson equation \eqref{pde} on the unit cube,  with exact
 solution 
\an{ \label{s-2}  u&=\sin(\pi x) \sin(\pi y)\sin(\pi z). } We use a wedge-type grids shown in
 Figure \ref{grid3}.  The
 lifted finite element solution has two orders of superconvergence,  shown in Table \ref{t3}.

\begin{figure}[h!]
\begin{center}
 \setlength\unitlength{0.8pt}
    \begin{picture}(320,118)(0,3)
    \put(0,0){\begin{picture}(110,110)(0,0)
       \multiput(0,0)(80,0){2}{\line(0,1){80}}  \multiput(0,0)(0,80){2}{\line(1,0){80}}
       \multiput(0,80)(80,0){2}{\line(1,1){20}} \multiput(0,80)(20,20){2}{\line(1,0){80}}
       \multiput(80,0)(0,80){2}{\line(1,1){20}}  \multiput(80,0)(20,20){2}{\line(0,1){80}}
    \put(80,0){\line(-1,1){80}}
      \end{picture}}
    \put(110,0){\begin{picture}(110,110)(0,0)
       \multiput(0,0)(40,0){3}{\line(0,1){80}}  \multiput(0,0)(0,40){3}{\line(1,0){80}}
       \multiput(0,80)(40,0){3}{\line(1,1){20}} \multiput(0,80)(10,10){3}{\line(1,0){80}}
       \multiput(80,0)(0,40){3}{\line(1,1){20}}  \multiput(80,0)(10,10){3}{\line(0,1){80}}
    \put(80,0){\line(-1,1){80}}
       \multiput(40,0)(40,40){2}{\line(-1,1){40}}
      \end{picture}}
    \put(220,0){\begin{picture}(110,110)(0,0)
       \multiput(0,0)(20,0){5}{\line(0,1){80}}  \multiput(0,0)(0,20){5}{\line(1,0){80}}
       \multiput(0,80)(20,0){5}{\line(1,1){20}} \multiput(0,80)(5,5){5}{\line(1,0){80}}
       \multiput(80,0)(0,20){5}{\line(1,1){20}}  \multiput(80,0)(5,5){5}{\line(0,1){80}}
    \put(80,0){\line(-1,1){80}}
       \multiput(40,0)(40,40){2}{\line(-1,1){40}}

       \multiput(20,0)(60,60){2}{\line(-1,1){20}}   \multiput(60,0)(20,20){2}{\line(-1,1){60}}
      \end{picture}}

    \end{picture}
    \end{center}
\caption{  The first three levels of wedge grids used in Table \ref{t3}. }
\label{grid3}
\end{figure}
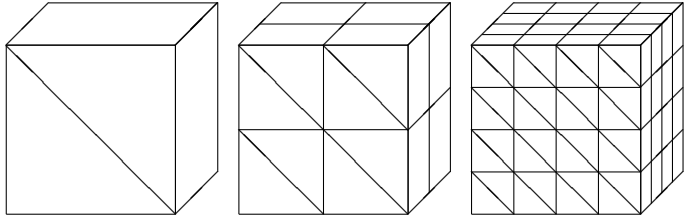

\begin{table}[ht]
  \centering   \renewcommand{\arraystretch}{1.05}
  \caption{ Errors and orders of convergence, by the $P_1$-$P_2$ WG finite element on 3D wedge-type polyhedral grids shown in Figure \ref{grid3} for \eqref{s-2}. }
\label{t3}
\begin{tabular}{c|cc|cc|cc}
\hline
level     & $\| u-  u_h \|_0 $  &rate &  $\| Q_h u- u_h \|_0 $ &rate  &
        $\| u- L_h u_h \|_0 $ &rate   \\ \hline
 4&    0.9655E-02&2.0&    0.1608E-03&3.9&    0.2626E-03&3.9 \\
 5&    0.2398E-02&2.0&    0.1022E-04&4.0&    0.1658E-04&4.0 \\
 6&    0.5987E-03&2.0&    0.6419E-06&4.0&    0.1039E-05&4.0 \\
\hline
& $ | u-  u_h  |_{1,h} $  &rate &  $\3bar Q_h u- u_h \3bar $ &rate &
        $| u- L_h u_h |_{1,h} $ &rate   \\
\hline
 4&    0.2289E+00&1.0&    0.2500E-01&3.0&    0.1269E-01&3.0 \\
 5&    0.1145E+00&1.0&    0.3136E-02&3.0&    0.1595E-02&3.0 \\
 6&    0.5724E-01&1.0&    0.3923E-03&3.0&    0.1997E-03&3.0 \\
\hline  
\end{tabular}%
\end{table}%

\end{document}